\documentclass[12pt]{article}
\usepackage{amsmath,amssymb,amsthm, comment}

\newtheorem{theorem}{Theorem}[section]
\newtheorem{thmy}{Theorem}

\newtheorem{lemma}[theorem]{Lemma}

\newcommand{\dd}{\displaystyle }

\def\barr{\begin{array}}
\def\earr{\end{array}}

\title{On a divisibility condition related to the sum of element orders of a finite group}
\author{Marius T\u arn\u auceanu}
\date{August 10, 2026}

\begin{document}

\maketitle

\begin{abstract}
Given a finite group $G$, we denote by $\psi(G)$ the sum of element orders of $G$. In this note, we determine finite groups $G$ such that $|H|-|K|$ divides $\psi(H)-\psi(K)$ for all subgroups $K\leq H\leq G$. Two weaker conditions are also proposed.
\end{abstract}

{\small
\noindent
{\bf MSC2020\,:} Primary 20D60; Secondary 20D15.

\noindent
{\bf Key words\,:} element orders, finite groups, ${\rm CP}_1$-groups.}

\section{Introduction}
Let $G$ be a finite group. In 2009, H. Amiri, S.M. Jafarian Amiri and I.M. Isaacs introduced in their paper \cite{1} the function
\begin{equation}
\psi(G)=\sum_{x\in G}o(x),\nonumber
\end{equation}where $o(x)$ denotes the order of $x$ in $G$. They proved the following basic theorem:

\begin{thmy}
If $G$ is a finite group of order $n$, then $\psi(G)\leq\psi(C_n)$, and we have equality if and only if $G$ is cyclic.
\end{thmy}Since then many authors have studied the properties of the function $\psi(G)$ and its relations with the structure of $G$ (see e.g. \cite{7}). Several divisibility conditions involving this function have been investigated in \cite{6,10,11}.

Note that Theorem A follows from the next result:

\begin{thmy}
If $G$ is a finite group of order $n$, then there is a bijection $f:G\longrightarrow C_n$ such that $o(x)$ divides $o(f(x))$, for all $x\in G$.
\end{thmy}

This has been formulated as a question by I.M. Isaacs (see Problem 18.1 in \cite{12}) and proved for some particular groups by F. Ladisch \cite{9} and M. Amiri and S.M. Jafarian Amiri \cite{2}. A proof for arbitrary groups has been recently given by M. Amiri \cite{3}. 

In the current note, we determine finite groups $G$ satisfying 
\begin{equation}
|H|-|K|\mid\psi(H)-\psi(K),\, \forall\, K\leq H\leq G.\footnote{This condition was inspired by the Mean Value Theorem in Mathematical Analysis.}
\end{equation}

Our main result is stated as follows.

\begin{theorem}
A finite group $G$ satisfies the condition (1) if and only if it is a $p$-group of exponent $p$.
\end{theorem}Its proof will be given in Section 2, while in Section 3 we will propose two conditions that are weaker than (1).

Most of our notation is standard and will usually not be repeated here. Elementary notions and results on groups can be found in \cite{8}.

\section{Proof of Theorem 1.1}

We start by recalling a well-known result that describes the class of ${\rm CP}_1$-groups, i.e. finite groups having all non-trivial elements of prime order.

\begin{lemma}[\cite{5,4}]
Let $G$ be a ${\rm CP}_1$-group. Then:
\begin{itemize}
\item[{\rm a)}] $G$ is nilpotent if and only if $G$ is a $p$-group of exponent $p$.
\item[{\rm b)}] $G$ is solvable and non-nilpotent if and only if $G$ is a Frobenius group with kernel $P\in Syl_p(G)$, with $P$ a $p$-group of exponent $p$ and complement $Q\in Syl_q(G)$, with $|Q|=q$. Moreover, if $|G|=p^nq$ then $G$ has a chief series 
\begin{equation}
G=G_0>P=G_1>G_2>\dots>G_k>G_{k+1}=1\nonumber 
\end{equation}such that for every $1\leq i\leq k$ one has $G_i/G_{i+1}\leq Z(P/G_{i+1})$, $Q$ acts irreducibly on $G_i/G_{i+1}$ and $|G_i/G_{i+1}|=p^b$, where $b$ is the exponent of $p\,\, {\rm (mod}\,\, q{\rm )}$.
\item[{\rm c)}] $G$ is non-solvable if and only if $G\cong A_5$.
\end{itemize}
\end{lemma}\newpage

We are now able to prove our main result. We recall the general formula for computing the value of $\psi$ for cyclic $p$-groups, i.e.:
\begin{equation}
\psi(C_{p^{\alpha}})=\dd\frac{p^{2\alpha+1}+1}{p+1}=p^{2\alpha}-p^{2\alpha-1}+p^{2\alpha-2}-...-p+1.\nonumber
\end{equation}

\noindent{\bf Proof of Theorem 1.1.} Let $G$ be a finite group satisfying the condition (1). Firstly, we prove that $G$ has no element of order $p^2$, for any prime $p$. Indeed, if there exists $a\in G$ with $o(a)=p^2$, then
\begin{equation}
p^2-1\mid\psi(\langle a\rangle)-\psi(1)=p(p-1)(p^2+1),\nonumber
\end{equation}that is 
\begin{equation}
p+1\mid p(p^2+1),\nonumber
\end{equation}a contradiction.
Secondly, we prove that $G$ has no element of order $pq$, for any distinct primes $p$ and $q$. Indeed, if there exists $b\in G$ with $o(b)=pq$, then
\begin{equation}
pq-p\mid\psi(\langle b\rangle)-\psi(\langle b^q\rangle)=(p^2-p+1)(q^2-q),\nonumber
\end{equation}that is 
\begin{equation}
p\mid (p^2-p+1)q,\nonumber
\end{equation}a contradiction. It follows that $G$ is a ${\rm CP}_1$-group.

Assume that $G$ is a group as in item b) of Lemma 2.1. Then it has elements of order $1$, $p$, and $q$. The elements of order $p$ are exactly the non-trivial elements of $P$, while the elements of order $q$ are the non-trivial elements of all conjugates of $Q$. Thus, $G$ has one element of order $1$, $p^n-1$ elements of order $p$, and $(q-1)p^n$ elements of order $q$, which leads to
\begin{equation}
\psi(G)=1+p(p^n-1)+q(q-1)p^n.\nonumber
\end{equation}Since $\psi(Q)=1+q(q-1)$, we obtain
\begin{equation}
q(p^n-1)=|G|-|Q|\mid\psi(G)-\psi(Q)=[p+q(q-1)](p^n-1),\nonumber
\end{equation}implying that
\begin{equation}
q\mid p+q(q-1),\nonumber
\end{equation}a contradiction.

Assume that $G\cong A_5$. Then $\psi(G)=211$ and taking a subgroup $H\leq G$ of order $2$, we get
\begin{equation}
58=|G|-|H|\mid\psi(G)-\psi(H)=208,\nonumber
\end{equation}a contradiction.\newpage

Thus $G$ is a $p$-group of exponent $p$. Note that for any two subgroups $K\leq H\leq G$, if $|H|=p^{\alpha}$ and $|K|=p^{\beta}$, then
\begin{equation}
p^{\alpha}-p^{\beta}=|H|-|K|\mid\psi(H)-\psi(K)=p^{\alpha+1}-p^{\beta+1}.\nonumber
\end{equation}This completes the proof.\qed

\section{Two weaker conditions}

Given a finite group $G$, the condition (1) can be weakened to the condition
\begin{equation}
|H|-1\mid\psi(H)-1,\, \forall\, H\leq G.
\end{equation}Such a group $G$ also does not have elements of order $p^2$ for any prime $p$, but it can have elements of order $pq$ for certain distinct primes $p$ and $q$. The smallest example with this property is the cyclic group $\mathbb{Z}_6$, which satisfies (2), but not (1).

The condition (2) can be also weakened to the condition
\begin{equation}
|G|-1\mid\psi(G)-1.
\end{equation}We remark that $C_8$ is the smallest group satisfying (3) but not (2), while the dicyclic group ${\rm Dic}_3$ of order $12$ is the smallest non-cyclic group exhibiting this behavior.

These classes of finite groups can be further explored. 
\bigskip

\noindent{\bf Acknowledgements.} The author is grateful to the reviewer for remarks which improve the previous version of the paper.
\bigskip

\noindent{\bf Funding.} The author did not receive support from any organization for the submitted work.
\bigskip

\noindent{\bf Conflicts of interests.} The author declares that he has no conflict of interest.
\bigskip

\noindent{\bf Data availability statement.} My manuscript has no associated data.

\vspace*{3ex}\small

\hfill
\begin{minipage}[t]{5cm}
Marius T\u arn\u auceanu \\
Faculty of  Mathematics \\
``Al.I. Cuza'' University \\
Ia\c si, Romania \\
e-mail: {\tt tarnauc@uaic.ro}
\end{minipage}

\end{document}